\documentclass[10pt, a4paper]{amsart}
\usepackage{amssymb, amsmath, amsfonts, amsthm, verbatim}
\usepackage{hyperref}
\newtheorem{theorem}{Theorem}[section]
\newtheorem{lemma}[theorem]{Lemma}
\newtheorem{definition}[theorem]{Definition}
\newtheorem{cor}[theorem]{Corollary}
\newtheorem{problem}[theorem]{Problem}
\newtheorem{question}[theorem]{Question}

\newtheorem{example}[theorem]{Example}

\newtheorem{prop}[theorem]{Proposition}

\title{Purely (Non-)Strongly Real Beauville Groups}

\author{Ben Fairbairn}
\address{Ben Fairbairn, Department of Economics, Mathematics and Statistics, Birkbeck, University of London, Malet Street, London WC1E 7HX, United Kingdom}
\email{b.fairbairn@bbk.ac.uk}

\begin{document}
\maketitle

\begin{abstract}
We discuss Beauville groups whose corresponding Beauville surfaces are either always strongly real or never strongly real producing several infinite families of examples.
\end{abstract}

\section{Introduction}

%\subsection{Main Definitions}

\begin{definition}
A surface $\mathcal{S}$ is a \textbf{Beauville surface} if
\begin{itemize}
\item the surface $\mathcal{S}$ is isogenous to a higher product, that is, $\mathcal{S}\cong(\mathcal{C}_1\times\mathcal{C}_2)/G$ where
$\mathcal{C}_1$ and $\mathcal{C}_2$ are algebraic curves of genus at least 2 and $G$ is a finite group acting
faithfully on $\mathcal{C}_1$ and $\mathcal{C}_2$ by holomorphic transformations in such a way that it
acts freely on the product $\mathcal{C}_1\times\mathcal{C}_2$, and
\item each $\mathcal{C}_i/G$ is isomorphic to the projective line $\mathbb{P}_1(\mathbb{C})$ and the covering map
$\mathcal{C}_i\rightarrow\mathcal{C}_i/G$ is ramified over three points.
\end{itemize}
\end{definition}

These surfaces were first defined by Catanese in \cite{C} and the first significant investigation of them was conducted by Bauer, Catanese and Grunewald in \cite{BCG}. They have numerous nice properties and are relatively easy to construct making them useful for producing counterexamples and testing conjectures. The following condition is also investigated in \cite{BCG}.

\begin{definition} Let $\mathcal{S}$ be a complex surface. We say that $\mathcal{S}$ is \textbf{strongly real} if there exists
a biholomorphism $\sigma\colon\mathcal{S}\rightarrow\overline{\mathcal{S}}$ such that $\sigma\circ\overline{\sigma}$ is the identity map.
\end{definition}

What makes these surfaces particularly easy to work with is that all of the above can be easily translated into group theoretic terms.

%\subsection{Group Theory}

\begin{definition}\label{MainDef} Let $G$ be a finite group. Let $x,y\in G$ and let
\[
\Sigma(x, y) :=\bigcup_{i=1}^{|G|}\bigcup_{g\in G}\{(x^i)^g,(y^i)^g,((xy)^i)^g\}.
\]

A \textbf{Beauville structure} for the group $G$ is a set of pairs of elements $\{(x_1, y_1),(x_2,y_2)\}\subset G\times G$
with the property that $\langle x_1, y_1\rangle = \langle x_2, y_2\rangle=G$ such that
\[
\Sigma(x_1, y_1)\cap \Sigma(x_2, y_2)=\{e\}.
\]
If $G$ has a Beauville structure we say that $G$ is a \textbf{Beauville group}.
\end{definition}

A group defines a Beauville surface if and only if it has a Beauville structure. Furthermore, the Beauville surface defined by a particular Beauville structure is strongly real if and only if the corresponding Beauville structure has the property of being strongly real that we define as follows.

\begin{definition}\label{SRDef}
 Let $G$ be a Beauville group and let $X =\{(x_1, y_1),(x_2, y_2)\}$ be a
Beauville structure for $G$. We say that $G$ and $X$ are \textbf{strongly real} if there exists an
automorphism $\phi\in\mbox{Aut}(G)$ and elements $g_i\in G$ for $i = 1, 2$
such that
\[
g_i\phi(x_i)g_i^{-1}=x_i^{-1}\mbox{ and }g_i\phi(y_i)g_i^{-1}=y_i^{-1}.
\]
\end{definition}

Most of the Beauville structures appearing in the literature are either explicitly shown to be strongly real or the question of reality is never pursued. In many groups there are elements that can never be inverted by automorphisms meaning any Beauville structure defined using such elements cannot be strongly real, indeed groups typically have numerous Beauville structures some of which are strongly real, some of which are not. Here we are interested in the extreme cases and thus make the following definition.

\begin{definition}\label{PureDef}
A finite group $G$ is a \textbf{purely strongly real Beauville group} if $G$ is a Beauville group such that every Beauville structure of $G$ is strongly real. A finite group $G$ is a \textbf{purely non-strongly real Beauville group} if $G$ is a Beauville group such that none of its Beauville structures are strongly real.
\end{definition}

Throughout we shall follow the conventions that in a group $G$ and elements $g,h\in G$ we have that $g^h=hgh^{-1}$ and $[g,h]=ghg^{-1}h^{-1}$.

This paper is organised as follows. In Section 2 we will give infinitely many examples of Beauville groups that have strongly real Beauville structures as well as Beauville structures that are not suggesting that groups typically lie in neither of the categories in Definition \ref{PureDef}. Despite this, we go on in Section 3 to give infinitely many examples of purely strongly real Beauville groups before in the final section giving infinitely many examples of purely non-strongly real Beauville groups.

\section{Neither Case}

We first prove results showing that Beauville groups typically fall into neither case with several different examples.

\begin{lemma}
If $n>5$, then the alternating group A$_n$ is neither a Purely strongly real Beauville group nor a non-strongly real Beauville group.
\end{lemma}

\begin{proof}
Strongly real Beauville structures for these groups were constructed by Fuertes and Gonz\'{a}lez-Diez in \cite{FG} so it is sufficient to construct non-strongly real Beauville structures for these groups. We consider the cases $n$ odd and $n$ even separately. Recall that in either case, if two permutations have different cycle type, then they cannot be conjugate.

First suppose that $n\geq7$ is odd. Let
\[
x_1:=(1,2,4)\mbox{ and }y_1:=(1,2,3,4,\ldots,n).
\]
We have that their product is the $n$-cycle $x_1y_1=(1,3,4,2,5,\ldots,n)$ by direct calculation. By considering the subgroup generated by the elements $x_1^{y_1^i}$ we have that the subgroup these elements generate is 2-transitive and therefore primitive. Since this subgroup contains the 3-cycle $x_1$ it follows that these elements generate the whole of the alternating group. It is easy to see that no automorphism simultaneously inverts both of these elements so any permutations that can be used to extend this to a Beauville structure gives a non-strongly real Beauville structure. Consider the permutations
\[
x_2:=(5,4,3,2,1)\mbox{ and }y_2:=(3,4,5,6,\ldots,n).
\]
We have that their product is the $(n-2)$-cycle $x_2y_2=(1,6,\ldots,n,3,2)$ by direct calculation. By considering the subgroup generated by the elements $x_2^{y_2^i}$ we have that the subgroup these elements generate is 2-transitive and therefore primitive. Since this subgroup contains the double-transposition $[x_2,y_2]=(1,5)(3,n)$ it follows that these elements generate the whole of the alternating group. We therefore have a non-strongly real Beauville group.

Next, suppose $n\geq8$ is even. Let
\[
x_1:=(1,2)(3,4)\mbox{ and }y_1:=(2,3,4,\ldots,n).
\]
We have that their product is the $(n-1)$-cycle $x_1y_1=(1,3,5,\ldots,n,2)$ by direct calculation. By considering the subgroup generated by the elements $x_1^{y_1^i}$ we have that the subgroup these elements generate is 2-transitive and therefore primitive. Since this subgroup contains the double-transposition $x_1$ it follows that these elements generate the whole of the alternating group. It is easy to see that no automorphism simultaneously inverts both of these elements so any permutations that can be used to extend this to a Beauville structure gives a non-strongly real Beauville structure.% Consider the permutations.

Finally, let
\[
x_2:=(1,2,3)(4,5,\ldots,n)\mbox{ and }y_2:=(5,4,3,2,1).
\]
We have that their product is the $(n-3)$-cycle $x_2y_2=(3,5,\ldots,n)$. By considering the subgroup generated by the elements $y_2^{x_2^i}$ we have that the subgroup these elements generate is 2-transitive and therefore primitive. Since this subgroup contains the double transposition $[x_2,y_2^2]=(2,5)(3,n)$ it follows that these elements generate the whole of the alternating group. We therefore have a non-strongly real Beauville group.

The case of $A_6$ is complicated by the existence of exceptional outer automorphisms but despite this can easily be handled separately.
\end{proof}

Further examples are given by the following.

\begin{lemma}
None of the Suzuki groups $^2B_2(2^{2n+1})$ are purely (non-)strongly real Beauville groups.
\end{lemma}

\begin{proof}
The Beauville structures constructed by Fuertes and Jones in \cite[Theorem 6.2]{FJ} are non-strongly real since they take $y_1$ as having order 4 and no automorphism maps such elements to their inverses in these groups. Strongly real Beauville structures for these groups were constructed by the author in \cite{F2}.
\end{proof}

\begin{lemma}
Aside from the Mathieu groups M$_{11}$ and M$_{23}$ none of the sporadic simple groups are purely (non-)strongly real Beauville groups.
\end{lemma}

\begin{proof}
Aside from M$_{11}$ and M$_{23}$, strongly real Beauville structures for these groups were constructed by the author in \cite{F1}. Non-strongly real Beauville structures are easily obtained computationally and with character theory, the more difficult cases being easily dealt with thanks to the existence of elements that cannot be sent to their inverses by any automorphism at all such as elements of order 71 in the monster group $\mathbb{M}$ and elements of order 47 in the baby monster group $\mathbb{B}$.
\end{proof}

We will return to the cases of M$_{11}$ and M$_{23}$ in Section \ref{NonSec} since they are genuine exceptions to the above.

It is clear that numerous further examples can be constructed from the above using direct products.

\section{Purely Strongly Real Beauville Groups}

Firstly, the following observation has been made elsewhere in the literature many times and provides infinitely many purely strongly real Beauville groups.

\begin{lemma}\label{AbelianProof}
A finite abelian group is a strongly real Beauville group if and only if it is a purely strongly real Beaville group.
\end{lemma}

\begin{proof}
In any abelian group the homomorphism $x\mapsto-x$ inverts every element.
\end{proof}

Abelian Beauville groups were first constructed by Catanese in \cite{C} and classified by Bauer, Catanese and Grunewald in \cite{BCG}. For non-soluble examples, we have the following infinite supply.

\begin{prop}\label{PruityL2(q)}
Let $q>2$ be a power of 2 and let $k$ be a positive integer. If $(q,k)\not=(4,1)$, then the characteristically simple group $L_2(q)^k$ is a purely strongly real Beaville group whenever it is 2-generated.
\end{prop}

\begin{proof}
In \cite{MacB} MacBeath observed that any generating pair for the groups $L_2(q)$ can be inverted by an inner automorphism when $q$ is even and since these groups have only one class of involutions, the only elements of even order.
\end{proof}

We remark that the exception $(q,k)=(4,1)$ is a genuine exception thanks to the isomorphism A$_5\cong\mbox{L}_2(4)$ and this group is well known to not be a Beauville group.

For each prime $p\geq5$ we can also construct infinitely many (new) non-abelian nilpotent examples as follows.

\begin{prop}\label{ExtraSpecProp}
If $p\geq5$ is a prime and $n\geq r\geq1$ are integers, then the group
\[
G:=\langle x,y,z\,|\,x^{p^n},\,y^{p^n},z^{p^r}\,[x,y]=z,\,[x,z],\,[y,z]\rangle
\]
is a purely strongly real Beauville group.
\end{prop}

\begin{proof}
Let $\{(x_1,y_1),(x_2,y_2)\}$ be a Beauville structure of $G$. We first note that Aut($G$) acts transitively on the non-central elements of $G$ of a given order and a generating pairs must consist of two element of order $p^n$, so without loss we may assume that $x_1=x$. Similarly $C_{\mbox{Aut}(G)}(x)$ acts transitively the elements $x$ can generate with so without loss we may also assume that $y_1=y$ and that $\phi\in$Aut($G$) acts by $x^\phi=x^{-1}$ and $y^{\phi}=y^{-1}$ (so we can take the element $g_1$ of Definition \ref{SRDef} to be trivial).

Observe the following. For any element $g\in G\setminus\Phi(G)$ we have that its conjugates are $g^G=\{gz^i\,|\,i=0,\ldots,p^r-1\}$. Every element of this group can be written in the form $x^iy^jz^k$ for some $0\leq i,j\leq p^n-1$ and $0\leq k\leq p^r-1$, so we have that $x_2=x^{i_1}y^{j_1}z^{k_1}$ and $y_2=x^{i_2}y^{j_2}z^{k_2}$ for some $1\leq i_1,i_2,j_1,j_2\leq p^n-1$ and $0\leq k_1,k_2\leq p^r-1$. Notice that if these are chosen so that $i_1\not=j_1$, $i_2\not= j_2$, $i_1+i_2\not=j_1+j_2$ and $o(x_2)=o(y_2)=p^n$, then these elements provide a Beauville structure and since $p\geq5$ finding such integers is straightforward. Moreover
\[
(x^iy^jz^k)^{\phi}=x^{-i}y^{-j}z^k\mbox{ and }(x^iy^jz^k)^{-1}=x^{-i}y^{-j}z^{-k-ij}
\]
We also have that for any $0\leq a,b\leq p^n-1$
\[
(x^{-i}y^{-j}z^k)^{x^ay^b}=x^{-i}y^{-j}z^{k+bi-aj}.
\]
It follows that to have $(x_2^{\phi})^{x^ay^b}=x_2^{-1}$ and $(y_2^{\phi})^{x^ay^b}=y_2^{-1}$ we must have that
\[
bi_1-aj_1\equiv-2k_1-ij\mbox{ (mod }p^r)\mbox{ and }bi_2-aj_2\equiv-2k_2-ij\mbox{ (mod }p^r).
\]
If the values of $i_1, i_2, j_1, j_2, k_1$ and $k_2$ are chosen so $x_2$ and $y_2$ generate the group and the conjugacy condition is satisfied, then values of $a$ and $b$ satisfying these equations can be found if $j_2i_1\not\equiv j_1i_2$ (mod $p$). We claim that if $j_2i_1\equiv j_1i_2$ (mod $p$), then $x_2$ and $y_2$ do not generate the group. Note that under this condition we have that
\[
(x^{i_1}y^{j_1})^{i_2}=x^{i_1i_2}y^{j_1i_2}z^k=x^{i_1i_2}y^{j_2i_1}z^k
\]
for some $k$ but we also have that
\[
(x^{i_2}y^{j_2})^{i_1}=x^{i_1i_2}y^{j_2i_1}z^{k'}
\]
for some $k'$ in other words $y_2\in\langle x_2,z\rangle$ which is a proper subgroup.
\end{proof}

We remark that strongly real Beauville $p$-groups have proved somewhat difficult to construct the only previously known examples being given in \cite{F3,F4,Gul1,Gul2}. The above provides infinitely many further new examples and does so for each prime $p\geq5$.

Clearly Proposition \ref{PruityL2(q)}, Lemma \ref{AbelianProof} and Proposition \ref{ExtraSpecProp} can be combined to produce infinitely many other examples like L$_2(8)\times C_5^2$ but it is not clear what other examples can arise.

\begin{problem}
Find other examples of purely strongly real Beauville groups.
\end{problem}

In particular, we have the following question.

\begin{question}
Do there exist purely strongly real Beauville 2-groups and 3-groups.
\end{question}

In the opinion of the author it seems likely that 2-generated 2-groups are more likely to be purely strongly real Beauville groups: there is a general philosophy in the study of $p$-groups that `the automorphism group of a $p$-group is typically a $p$-group' thanks to the results of Helleloid and Martin \cite{HM}. In particular, if $p$ is odd, then typically no automorphism like the $\phi$ in of Definition \ref{SRDef} exists since such an automorphism must necessarily have even order. On the other side of the coin however, most of their results give examples with large numbers of generators and examples that are 2-generated seem difficult to construct.

\section{Purely non-strongly Beauville groups}\label{NonSec}

What about purely non-strongly real Beauville groups? In \cite[Lemma 2.2]{F1} the author shows, in the terminology defined here, that the Mathieu groups M$_{11}$ and M$_{23}$ are purely non-strongly real Beauville groups, indeed the real content of \cite[Conjecture 1]{F2} is that among the non-abelian finite simple groups these are really the only ones. For an infinite supply of examples we have the following.

\begin{prop}\label{nonSR}
If $G$ and $H$ are Beauville groups of coprime order, such that $G$ is a purely non-strongly real Beauville group, then $G\times H$ is a purely non-strongly real Beauville group.
\end{prop}

\begin{proof}
Let $\{(x_1,y_1),(x_2,y_2)\}$ be a Beauville structure for $G$ and let $\{(u_1,v_1),(u_2,v_2)\}$ be a Beauville structure for $H$. Since $|G|$ and $|H|$ are coprime it follows that
\[
\{((x_1,u_1),(y_1,v_1)),((x_2,u_2),(y_2,v_2))\}
\]
is a Beauville structure for $G\times H$. Since $|G|$ is coprime to $|H|$ it follows that $G$ and $H$ are not isomorphic and so Aut($G\times H$)=Aut($G)\times$Aut($H$). Since elements of Aut($G$) cannot be used to provide a strongly real Beauville structure, it follows that none of the Beauville structures of $G\times H$ are strongly real.
\end{proof}

\begin{cor}
There exist infinitely many purely non-strongly real Beauville groups.
\end{cor}

\begin{proof}
As noted above, $M_{11}$ is a purely non-strongly real Beauville group. In Proposition \ref{nonSR} we can therefore take $G$ to be M$_{11}$. Since $|\mbox{M}_{11}|=11\times5\times3^2\times2^4$ we can take $H$ to be any Beauville group has order coprime to 11, 5, 3 and 2. Infinitely many examples of such groups are constructed in \cite{BBF,F4,Gul1,Gul2,SV} as well as in Proposition \ref{ExtraSpecProp}.
\end{proof}

The next example shows that Proposition \ref{nonSR} is far from being the best possible.

\begin{example}
The group $M_{11}\times A_5$ is easily seen to be a Beauville group as the elements
\[
x_1:=(1,2,3,4,5,6,7,8,9,10,11)(12,13,14,15,16),
\]
\[
y_1:=(1,5,3,4,10,2,8,9,11,6,7)(12,14,15,13,16)
\]
along with
\[
x_2:=(1,2,9,10,6)(3,11,5,4,7)(12,13,14,15,16),
\]
\[
y_2:=(1,4,8,11,3)(2,9,7,5,6)(12,14,15,13,16)
\]
provide a Beauville structure of type ((55,55,55),(5,5,5)). The lack of automorphisms that make $M_{11}$ a purely non-strongly real Beauville group clearly also make $M_{11}\times A_5$ a purely non-strongly real Beauville group. Note that every prime dividing the order $A_5$ also divides the order of $M_{11}$ and that $A_5$ not even a Beauville group.
\end{example}

\begin{question}
What is the most general form of Proposition \ref{nonSR}?
\end{question}

The author is not aware of any nilpotent or even soluble examples. As previously mentioned `the automorphism group of a $p$-group is typically a $p$-group'. Among the groups of order $p^n$ for small $n$ few examples of 2-generated groups with an automorphism group of odd order exist and the few that do appear to not be Beauville groups, suggesting that such examples are actually quite difficult to construct.

\begin{question}
Do there exist any nilpotent or soluble purely non-strongly real Beauville groups?
\end{question}

This would be immediately answered by combining Proposition \ref{nonSR} with an answer to the following.

\begin{question}
Do there exist Beauville $p$-groups whose automorphism groups have odd order?
\end{question}


\begin{thebibliography}{99}

\bibitem{BBF} N.\,W. Barker, N. Bosten and B.\,T. Fairbairn ``A note on Beauville $p$-groups" Exp.
Math., 21(3): 298--306 (2012) \href{http://arxiv.org/abs/1111.3522}{\texttt{arXiv:1111.3522v2}} doi:10.1080/10586458.2012.669267

\bibitem{BCG} I. Bauer, F. Catanese and F. Grunewald, Beauville surfaces without real structures,
in \emph{Geometric methods in algebra and number theory} pp. 1--42, \emph{Progr. Math.}, 235,
Birkh\"{a}user Boston, Boston, MA, 2005.

\bibitem{C} F. Catanese, Fibered surfaces, varieties isogenous to a product and related moduli
spaces, \emph{Amer. J. Math.} 122 (2000), no. 1, pp. 1--44

\bibitem{F1} B.\,T. Fairbairn ``Some Exceptional Beauville Structures" J. Group Theory, 15(5), pp.
631--639 (2012) \href{http://arxiv.org/abs/1007.5050}{\texttt{arXiv:1007.5050}} \texttt{DOI: 10.1515/jgt-2012-0018}

\bibitem{F2} B.\,T. Fairbairn, ``Strongly Real Beauville Groups" in \emph{Beauville Surfaces and Groups, Springer Proceedings in Mathematics \& Statistics, Vol. 123} (eds I. Bauer, S. Garion and
A. Vdovina), Springer-Verlag (2015) pp. 41--61

\bibitem{F3} B.\,T. Fairbairn, More on Strongly Real Beauville Groups, in \emph{Symmetries in Graphs Maps and Polytopes 5th SIGMAP Workshop, West Malvern, UK, July 2014} (eds. J. \v{S}ir\'{a}\v{n} and R. Jajcay) Springer Proceedings in Mathematics \& Statistics  159 (2016) pp. 129--146

\bibitem{F4} B.\,T.  Fairbairn, A new infinite family of non-abelian strongly real Beauville $p$-groups for every odd prime $p$, \emph{Bull. Lond. Math. Soc.}  49(4) (2017) doi:10.1112/blms.12060  \href{http://arxiv.org/abs/1608.00774}{\texttt{arXiv:1608.00774}}

%\bibitem{F5} B. T. Fairbairn, A Corrigendum to ``A new infinite family of non-abelian strongly real Beauville $p$-groups for every odd prime $p$", in preparation


\bibitem{FG} Y. Fuertes and G. Gonz\'{a}lez-Diez ``On Beauville structures on the groups $S_n$ and $A_n$"
Math. Z. 264 (2010), no. 4, 959--968

\bibitem{FJ} Y. Fuertes and G. Jones ``Beauville surfaces and finite groups" J. Algebra 340 (2011)
13--27

\bibitem{Gul1} \c{S}. G\"{u}l, A note on strongly real Beauville $p$-groups, \emph{Monatsh. Math.} (2017)\\ doi:10.1007/s00605-017-1034-1 \texttt{arXiv:1607.08907}

\bibitem{Gul2} \c{S}. G\"{u}l, An infinite family of strongly real Beauville $p$-groups, preprint 2016\\ \href{https://arxiv.org/pdf/1610.06080.pdf}{\texttt{arXiv:1610.06080}}


\bibitem{HM} G.\,T. Helleloid and U. Martin, The automorphism group of a finite $p$-group is almost always a $p$-group, \emph{J. Algebra} 312 (2007) pp. 284--329

\bibitem{MacB} A. M. MacBeath ``Generators of the linear fractional groups" Number Theory (Proc. Sympos. Pure Math., Vol. XII, Houston, Tex., 1967), Amer. Math. Soc., Providence, R.I., 1969, pp. 14--32

\bibitem{SV} J. Stix and A. Vdovina, Series of $p$-groups with Beauville structure,  \emph{Monatsh. Math.} 181 (2016), no. 1, pp. 177--186 doi:10.1007/s00605-015-0805-9 \texttt{arXiv:1405.3872}

\end{thebibliography}
\end{document}